\documentclass[11pt,reqno]{article}
\usepackage{amsmath,amsthm,amsfonts,amssymb,amscd}
\usepackage[latin1]{inputenc}

\usepackage[latin1]{inputenc}

\usepackage{amssymb}
\theoremstyle{plain}
\newtheorem{Theorem}{Theorem}
\newtheorem{Corollary}{Corollary}
\newtheorem{Example}{Example}

\newtheorem{Lemma}{Lemma}
\newtheorem{Proposition}{Proposition}
\newtheorem{Question}{Question}

\theoremstyle{Definition}
\newtheorem{Definition}{Definition}
\theoremstyle{Remark}
\newtheorem{Remark}{Remark}

\numberwithin{equation}{section}

\def\fa{{\mathcal{F}}}

\def\po{{\partial}}

\def\vr{{\varphi}}
\def\ga{{\gamma}}

\def\la{{\lambda}}

\def\ov{\overline}

\def\lg{{\langle}}
\def\rg{{\rangle}}

\def\re{{\mathbb{R}}}

\def\bc{{\mathbb{C}}}

\def\Sim{\operatorname{{Sim}}}

\def\Re{\operatorname{{Re}}}

\def\dim{\operatorname{{dim}}}

\def\Ker{\operatorname{{Ker}}}

\def\sing{\operatorname{{sing}}}
\def\Sing{\operatorname{{sing}}}
\def\cod{\operatorname{{cod}}}
\def\grad{\operatorname{{grad}}}

\def\leaderfill{\leaders\hbox to .8em{\hss .\hss}\hfill}
\def\_#1{{\lower 0.7ex\hbox{}}_{#1}}

\def\fa{{\mathcal{F}}}

\def\po{{\partial}}

\def\vr{{\varphi}}
\def\ga{{\gamma}}

\def\la{{\lambda}}

\def\ov{\overline}

\def\lg{{\langle}}
\def\rg{{\rangle}}

\def\re{{\mathbb{R}}}

\def\bc{{\mathbb{C}}}

\def\Re{\operatorname{{Re}}}

\def\dim{\operatorname{{dim}}}

\def\Ker{\operatorname{{Ker}}}

\def\sing{\operatorname{{sing}}}
\def\Sing{\operatorname{{Sing}}}

\def\grad{\operatorname{{grad}}}

\begin{document}
\title{A non-existence theorem for Morse type holomorphic foliations of codimension one
transverse to spheres}
\author{T. Ito and B. Sc\'ardua}

\date{}
{\tiny \maketitle }

\begin{abstract}
We prove that a Morse type codimension one holomorphic foliation
is not transverse to a sphere in the complex affine space. Also we
characterize the variety of contacts of a linear foliation with
concentric spheres.
\end{abstract}


\pagenumbering{arabic}

\section{Introduction} \label{Section:introduction}

One of the main tools in the classical theory of codimension one
real foliations  is a theorem of   A. Haefliger
\cite{Camacho-LinsNeto}
 which implies that an analytic codimension-one foliation admits no
null (homotopic) transversals. The heart of the proof consists of
a description of the dynamics of a real vector field in a
neighborhood of the closed disc $\ov{D^2} \subset \re^2$ and
transverse to the boundary $\po D^2 \simeq S^1$. The use of
Poincar\'e-Bendixson Theorem  shows the existence of some
one-sided hyperbolicity,  for some closed orbit or graph $\ga
\subset D^2$, what is not compatible with the analytical behavior.
Unfortunately, there is no feature like the classical
Poincar\'e-Bendixson Theorem in the case of holomorphic vector
fields. To overcame this difficult is one of the two basic
motivations for the present work. The second comes from \cite{GSV}
where  the authors study the topology of the {\it variety of
contacts} of a holomorphic vector field $F=\sum\limits_{j=1}^n F_j
\frac{\partial}{\partial z_j}$ in an open neighborhood $\mathcal
U$ of the origin $0\in \bc^n$, having an isolated singularity at
$0$, with the spheres around $0$. Let $M$ be this variety of
contacts, i.e., $M=\{ z\in \mathcal U \big| \lg F(z), z \rg
=:\sum\limits_{j=1} ^n F_j(z) \ov z_j= 0\}$. In case the
singularity is in the {\it Siegel domain} we have $M \ne 0$ and
results in \cite{Arnold}, \cite{Camachoetal} and \cite{Guck} imply
that if $F$ is linear and generic then $M\setminus \{0\}$ is a
smooth codimension two manifold in $\bc ^n$, each {\it Siegel
leaf}  of $F$ intersects $M\setminus \{0\}$ in exactly one point,
which corresponds to the minimal distance of the leaf to the
origin. The main result in \cite{GSV} then generalizes some of
these properties to a class of vector fields called of {\it Morse
type}, corresponding to vector fields for which the distance
function on each leaf has only nondegenerate singularities.
Indeed, the class corresponds to those for which the {\it distance
flow}, which is the real analytic vector field defined by $r_F:=-
\ov{\lg F(z), z\rg } F(z)$, has only nondegenerate singularities
on each leaf. Using this notion the authors  are able to prove
that for a Morse type vector field $F$ either $M=0$ or $M$ is a
codimension two real analytic variety, singular only at $0$,
indeed $M$ is a cone and $M\setminus \{0\}$ embeds into $\bc^n$
with trivial normal bundle (it is transversal to the foliation of
$F$), has only a finite number of connected components, and
exhibits some stability with respect to the induced foliations on
the small spheres around the origin.

In this paper we shall begin the  study of the variety of contacts
for a codimension one holomorphic foliation in complex dimension
$n \geq 3$. This is also done by introducing a class of {\it Morse
type} foliations and we prove that for such a foliation we never
have $M=0$. In other words, we give a negative answer for the
following question in case of Morse type foliations:

\begin{Question}
Is there a holomorphic codimension one foliation transverse to the
unit sphere $S^{2n-1}(1)\subset \bc^n$ in $\bc^n$ if $n \geq 3$?
\end{Question}

As mentioned above this question is related to the dynamics of
holomorphic foliations of codimension one and it seems that the
answer is no. In the last part we give a description of the
variety of contacts for generic linear one-forms.

Let us make precise the notions that we use. Let $\fa$ be a
holomorphic foliation, possibly with singularities, on a complex
manifold $V$. Given a smooth (real) submanifold $M\subset V$ we
shall say that $\fa$ is {\em transverse} to $M$ if $\sing(\fa)\cap
M = \emptyset$ and for every $p \in M$ we have $T_p(L_p) + T_p(M)
= T_p(V)$ as real spaces, where $L_p$ denotes the leaf of $\fa$
that contains the point $p \in M$. In this paper we address
Question 1. As it is known such foliation is defined by a
holomorphic one-form $\Omega$, with singular set of codimension
$\geq 2$, and satisfying the integrability condition $\Omega
\wedge d \Omega = 0$.  In  previous papers \cite{Ito-ScarduaMMJ},
\cite{Ito-Scarduamrl} and \cite{Ito-Scarduatopology} we proved the
non-existence of the foliation under some additional conditions:
for instance, $\fa$ is defined by an integrable homogeneous
polynomial one-form, or $\fa$ has a global separatrix transverse
to each sphere $S^{2n-1}(r), \, 0 < r \leq 1$, or $n$ is odd. In
the case of a non-integrable holomorphic one-form, if $n=2m+1$ is
odd there is no holomorphic one-form $\Omega$ such that the
distribution $\Ker (\Omega)=\{\Omega=0\}$ is transverse to the
sphere $S^{4m+1}(1) \subset \bc ^{2m+1}$ (\cite{Ito-ScarduaMMJ}).
If $n$ is even, we have a typical example
(\cite{Ito-Scarduatopology}). In $\bc^{2m}$ with coordinates
$(z_1,...,z_{2m})$ we define $\Omega=\sum\limits_{j=1}^{2m}
(z_{2j} dz_{2j-1} - z_{2j-1} dz_{2j})$. This holomorphic one-form
satisfies the strong non-integrability condition $\Omega \wedge d
\Omega \ne 0$, off the origin, and $\Ker(\Omega)$ is transverse to
the sphere $S^{4m-1}(1) \subset \bc^{2m}$.

On the other hand, in the case of  a holomorphic vector fields on
$\bc^n, \, n \geq 2$, a Poincar\'e-Bendixson type theorem has been
proved in \cite{[Ito]} stating that if  $Z$ is  a holomorphic
vector field in a neighborhood $U$ of the closed unit disk
$\ov{D^{2n}} \subset \bc ^n, n \geq 2$, such that the
corresponding foliation $\fa(Z)$ the foliation defined by the
solutions of $Z$ is transverse to  $S^{2n-1}(1) =
\partial \ov{D^{2n}}$, then $Z$ has only
one singular point, say $p$, in $\ov{D^{2n}}$ and the index of $Z$
at $p$ is equal to one. Moreover, each solution $L$ of $Z$ which
crosses $S^{2n-1}(1)$ tends to the unique singular point $p$ of
$Z$ in the disk, {\it i.e.}, $p$ is in the closure of $L$.
Furthermore, the restriction $\fa(Z)\big|_{\ov{D^{2n}}\setminus
\{p\} }$ is $C^w$-conjugate to the foliation
$\fa(Z)\big|_{S^{2n-1}(1)} \times (0,1]$ of $\ov{D^{2n}}\setminus
\{p\}$.

This paper is dedicated to the proof of non-existence in case of
generic foliations, where ``generic" stands for a generic set
tangencies with the spheres centered at the origin as follows. Let
$\fa(\Omega)$ be as above and denote by $\vr$ the distance
function with respect to the origin $0 \in \bc^n$.
\begin{Definition}[Morse type]
{\rm The foliation $\fa=\fa(\Omega)$ is of {\it Morse type} if for
each leaf $L\in \fa (\Omega)$, each critical point $p \in L$ of
$\vr\big|_L$ on $L$ is  nondegenerate.  $\Omega$ is of {\it Morse
type} if $\fa(\Omega)$ is.}
\end{Definition}

Our main result is the following non-existence theorem:

\begin{Theorem}
\label{Theorem:main} There is no  Morse type holomorphic foliation
of codimension one  $\fa$ in a neighborhood $U$ of the closed unit
disk $\ov{D^{2n}}\subset \bc^n, \, n \geq 3$ such that $\fa$ is
transverse to the boundary sphere $S^{2n-1}(1)=\partial
\ov{D^{2n}}$.
\end{Theorem}

A {\it linear foliation} is defined by a  one-form $\Omega
_A=\sum\limits_{i,j=1}^n a_{ij} z_i dz_j$, where
$A=(a_{ij})_{i,j=1}^n$ is a symmetric matrix. Regarding
transversality of linear foliations with spheres we prove in
\cite{Ito-ScarduaMMJ} that a linear foliation  $\fa(\Omega_A)$ on
$\bc^n$ is not transverse to the sphere $S^{2n-1}(1)$. Moreover,
$\fa(\Omega_A)$ is transverse to the sphere $S^{2n-1}(1)$ off the
singular set $\sing(\fa(\Omega_A))\cap S^{2n-1}(1)$ if and only if
$\fa(\Omega_A)$ is a product ${\mathcal L}_\la \times \bc^{n-2}$
for some linear foliation ${\mathcal L}_\la  \colon x\,dy - \la
y\,dx = 0$, in the Poincar\'e domain on $\bc^2$.

The following result then describes the variety of contacts for a
generic linear foliation.

\begin{Theorem}
\label{Theorem:linearapprox} Any linear foliation  $\fa$  on
$\bc^n$ can be arbitrarily approximated by Morse type linear
foliations. In particular, $\fa$ is not transverse to
$S^{2n-1}(1)$ if $n \geq 3$.  For a dense Zariski subset of linear
foliations on $\bc^n$ the variety of contacts of its elements is a
union of $n$ complex lines through the origin.
\end{Theorem}

\section{Pugh's generalization of Poincar\'e-Hopf theorem}
\label{section:Pugh}

In this section we recall Pugh's generalization of Poincar\'e
index formula (\cite{Pugh}) and we given an example which we use
in \S ~\ref{section:proofofmain} Lemma~\ref{Lemma:indexk}.

Let $X$ be a $C^\infty$ vector field on a manifold $M$, such that:
\begin{enumerate}

\item  $X$ has no singularity on $\partial M$.

\item The singularities of $X$ in $M$ are of Morse type.

\item $X$ has a generic contact set with $\partial M$.

\end{enumerate}

Let $n_i$ denote the number of zeros of $X$ of Morse index $i$,
where the Morse index is the number of eigenvalues of $DX(p)$
having negative real part. Let also $\Sigma(X):=\sum\limits_{i}
(-1)^i n_i$. Then we have
\[
\Sigma(X)=\mathcal X(M, \partial M) + \sum\limits_{i \geq 1}
\mathcal X (R^i _-, \Gamma^i)
\]
where
\begin{enumerate}
\item $R^i _-$ is the i-codimensional exit region.

\item $\Gamma^i$ is the boundary of $R_- ^i$.

\item The Euler characteristic is defined as
\[
\chi\mathcal (M,
\partial M)=\sum\limits_{i} \dim H_i(M,\partial M),
\]
 in terms of relative homology.

\end{enumerate}

If $M=D^2$, the unit disk in $\bc$,  we have $\mathcal X
(M,\partial M)=1$, $R_- ^i$ is the set of arcs on $\partial
D^2=S^1$ where $X$ points out of $D^2$, $\Gamma^1$ is the set of
tangency points, $R^2 _-$ is the set of points on $\partial
R^1_-=\Gamma^1$ where $X$ points out of $R_- ^1$. Because of the
generic contact between $X$ and $\partial M$, there are no
inflection points on $S^1$ and therefore, exit and entrance arcs
alternate around $S^1$. Therefore $\mathcal X(R^1 _-)$ is half the
number of tangency points, i.e., $\mathcal X (\Gamma^1)/ 2$. $R^2
_-$ is easily seen to be the set of interior tangencies and
$\Gamma^2 = \emptyset$. We have then
\[
\mathcal X(M,\partial M) + \sum\limits_{i \geq 1} \mathcal X (R^i
_-, \Gamma^i)= 1 + \mathcal X(R ^1_-, \Gamma^1) + \mathcal X
(R^2_-, \emptyset)
\]
\[  = 1 + \mathcal X(\Gamma ^1) / 2  -
\mathcal X (\Gamma^1) + i = 1 + \frac{ 2 i - \mathcal X
(\Gamma^i)}{2} = 1 + \frac{2i - (i + e)}{2} = 1 + \frac{i  -
e}{2}.
\] We obtain therefore the formula below quoted as {\em Poincar\'es
formula}.
\[
I=1 + \frac{ i - e}{2}
\] where $I$ is  the topological index of the vector field, $i$ is the
number of interior tangencies
and $e$ is the number of exterior tangencies.

\begin{Example}
{\rm Consider the case of an isolated Morse singularity of index
$i \geq 1$. Let $f(x)=- x_1 ^2 - ...-x_i ^2 + x_{i+1} ^2 +...+ x_n
^2$ be the Morse function of Morse index $i$ at $0$. We denote by
$\grad(f)=-2x_1 \frac{\partial}{\partial x_1} -...-  2x_i
\frac{\partial}{\partial x_i} + 2 x_{i+1} \frac{\partial}{\partial
x_{i+1}} + ...+ 2x_n \frac{\partial}{\partial x_n}$ the gradient
vector field of $f$ and by $\vec n=\sum\limits_{i=1}^n x_i
\frac{\partial}{\partial x_i}$the radial vector field on $\mathbb
R^n$. Take $r>0$, we investigate the gradient $\grad(f)$ in the
disc $D^n (r)$. Then we have $R_- ^1=\{x \in S^{n-1}(r) \big| \lg
\grad(f)(x), \vec n (x) \rg >0\}$ and $\Gamma^1 = \{ x \in
S^{n-1}(r) \big| \lg \grad(f)(x), \vec n (x) \rg =0\}$. Therefore,
$R_- ^1$ is homotopic to $D^i \times S^{n-i - 1}\cong S^{n-i-1}$,
where $D^i$ is the $i$-dimensional disc, $S^{n-i-1}$ the boundary
of $D^{n-i}$ and $\cong$ means homotopic to. On the other hand
$\Gamma^1$ is homotopic to $S^{i-1} \times S^{n-i-1}$. Then we can
calculate $\chi(R_- ^1, \Gamma^1)$ as follows:

\[
\chi(R_- ^1, \Gamma^1) = \chi(R_- ^1) - \chi(\Gamma^1) =
\chi(S^{n-i-1}) - \chi(S^{i-1}).\chi(S^{n-i-1}).
\]

We will apply Pugh's generalization of Poincar\'e-Hopf theorem to
the case of leaves of dimension $n=2(m-1)$:

\[
(-1)^i . 1 = 1 + \begin{cases} 2 - 2 \times 2, \,  \, \, \text{\rm
if $i$ is odd}
\\  0 - 0 \times 0, \, \, \,  \text{\rm if $i$ is even.}
\end{cases}
\]

}
\end{Example}

\section{The variety of contacts}

\label{section:varietyofcontacts}

Let $\Omega= \sum\limits_{j=1}^n f_j(z)dz_j$ be a holomorphic
one-form in a neighborhood $U$ of the origin $0\in\bc^n, n \geq
2$. We do not assume the integrability condition.
 We define the {\it gradient of\/} $\Omega$ as
the complex $C^\infty$ vector field $\grad(\Omega) = \sum_{j=1}^n
\ov{f_j(z)}\, \frac{\po}{\po z_j}\,\cdot$
 By construction, $\grad(\Omega)$ is orthogonal to the
distribution $\Ker (\Omega)$,  and also $\Omega\cdot\grad(\Omega)
= \sum\limits_{j=1}^n |f_j(z)|^2$. Let also $\vr$ be a real
analytic function $\vr \colon \bc^n \to \re$ and $\vr\ne 0$ in
$\bc^{n}\setminus \{0\}$.
 The {\it gradient vector field} of $\vr$ is defined by
 \begin{equation}
\grad(\vr)=2\big[\sum\limits_{j=1}^n (\ov{(\frac{\partial
\vr}{\partial z_j})}\frac{\partial}{\partial z_j} +
\ov{(\frac{\partial \vr}{\partial \ov
z_j})}\frac{\partial}{\partial \ov z_j}\big]= 4
\Re(\sum\limits_{j=1}^n (\ov{(\frac{\partial \vr}{\partial
z_j})}\frac{\partial}{\partial z_j}).
 \end{equation}

 Given $r>0$ let $S^{2n-1}(0;r)
\subset \bc^n$ be the sphere given by $\sum\limits_{j=1}^n |z_j|^2
= r^2$ and  $\vec R = \sum\limits_{j=1}^n z_j\frac{\po}{\po z_j}$
the {\it {\rm(}complex{\rm)} radial vector field}. If we write
$z_j = x_j + \sqrt{-1}y_j$ in standard euclidian coordinates then
the real normal vector field to $S^{2n-1}(0;r)$ is $\vec n :=
\sum\limits_{j=1}^n \big(x_j\,{\po}/{\po x_j} + y_j\,{\po}/{\po
y_j}\big)$. Write  $f_j(z)= g_j(x,y) + \sqrt{-1}h_j(x,y), \, 1
\leq j \leq n$. We have the real representations of these gradient
vector fields:

\begin{equation}
\grad(\Omega) = \frac{1}{2}
\{\sum_{j=1}^n\big(\left(g_j\,{\po}/{\po x_j} - h_j\,{\po}/{\po
y_j}\right) - \sqrt{-1} \sum_{j=1}^n\big(\left(h_j\,{\po}/{\po
x_j} + g_j\,{\po}/{\po y_j}\right)\}
\end{equation}
\begin{equation}
 = \frac{1}{2} \{ X -
\sqrt{-1} JX\}.
\end{equation}
where $X=\sum_{j=1}^n (g_j\,{\po}/{\po x_j} - h_j\,{\po}/{\po
y_j}),$ and $J$ is the canonical complex structure of $\bc^n$.
Also we have
\begin{equation}
\grad(\vr)= 4\Re\{\frac{1}{4}(\sum\limits_{j=1}^n (\frac{\partial
\vr}{\partial x_j}\frac{\partial}{\partial x_j} + \frac{\partial
\vr }{\partial y_j}\frac{\partial }{\partial y_j}) -
\sqrt{-1}(\sum\limits_{j=1}^n ( - \frac{\partial \vr}{\partial
y_j}\frac{\partial}{\partial x_j} + \frac{\partial \vr }{\partial
x_j}\frac{\partial}{\partial y_j}))\}
\end{equation}
\begin{equation}
= \Re( \vec n_\vr - \sqrt{-1} J\vec n_\vr) =  \vec n_\vr
\end{equation}
where by definition we have
\begin{equation}
\vec n _ \vr:=\sum\limits_{j=1}^n (\frac{\partial \vr}{\partial
x_j}\frac{\partial}{\partial x_j} + \frac{\partial \vr }{\partial
y_j}\frac{\partial }{\partial y_j}) \end{equation}

\begin{Definition}[Variety of contacts]
{\rm The {\it variety of contacts $\Sigma=\Sigma(\Omega, \vr)$ of
the one-form $\Omega$ and the function $\vr$}, i.e., the variety
of contacts $\Sigma$ of the two foliations $\fa(\Omega)$ and
$\fa(d \vr)$, is defined by the real analytic equations:
$\grad(\vr)(p) \in \grad(\Omega)(p)$ at $p \in \Sigma$, that is,
$\vec n_\vr = a X + b JX,$ \,  for some $a, b\in \re$ at $p \in
\Sigma$. }
\end{Definition}

\begin{Remark}
{\rm With the above definitions we have:

\begin{enumerate}

\item  Consider the distributions $\fa(\Omega)$ defined by
$\Omega$ and $\fa(d \vr)$ defined by the real one-form $d \vr$.
The relation between two tangent spaces of these distributions is
$T_p(\fa(\Omega))\subset T_p(\fa(d \vr))$ at $p \in \Sigma$.

\item Put $\vec R_\vr = \frac{1}{2}(\vec n_\vr - \sqrt{-1} J \vec
n_\vr) =2 \sum\limits_{j=1}^n \ov{(\frac{\partial \vr}{\partial
z_j})} \frac{\partial}{\partial z_j}$. Then, $\Sigma(\Omega, \vr)$
is defined by the  equations:

\begin{equation}
\label{equation:3.7} \frac{f_1(z)}{\frac{\partial\vr}{\partial
z_1}}=\frac{f_2(z)}{\frac{\partial\vr}{\partial
z_2}}=...=\frac{f_n(z)}{\frac{\partial\vr}{\partial z_n}}
\end{equation}

\end{enumerate}
\begin{proof}
 (1) Given $v \in T_p(\fa(\Omega))$ we have $<v, X>=0$ and $<v, JX>=0$.
This implies $<v,\vec n  _ \vr>= a<v,X> + b<v, JX>=0$.

(2) By definition and the above equations we have
\begin{equation}
\vec R_\vr = \frac{1}{2}( a X + b JX - \sqrt{-1}(a JX - b X)) = (
a + \sqrt{-1} b ) \frac{1}{2} ( X - \sqrt{-1} JX)_ = ( a +
\sqrt{-1} b ) . \grad (\Omega)
\end{equation}

Thus $\vec R_\vr = \lambda . \grad(\Omega)$ for some $\lambda \in
\bc$, i.e., (\ref{equation:3.7}).

\end{proof}

}
\end{Remark}

\subsection{The distance function}
\label{subsection:distancefunction}  Let now
$\vr(z)=\sum\limits_{j=1}^n |z_j|^2$ be the distance function from
the origin to the point $z=(z_1,...,z_n)\in \bc^n$. Then
\begin{equation}
\grad(\vr) =  \vec n _\vr= 2 \vec n.
\end{equation}

where
\begin{equation}
\vec n := \sum\limits_{j=1}^n \big(x_j\,{\po}/{\po x_j} +
y_j\,{\po}/{\po y_j}\big)
\end{equation}

In particular, the variety of contacts $\Sigma$ is given by
\begin{equation}
\frac{f_1(z)}{\ov{z_1}}=\frac{f_2(z)}{\ov{z_2}}=...=\frac{f_n(z)}{\ov{z_n}}
\end{equation}

\subsection{Homogeneous case}
\label{subsection:homogeneouscase}
\begin{Proposition}
Let $\Omega=\sum\limits_{j=1}^n f_j(z) dz_j$ be an integrable
homogeneous polynomial one-form of degree $k \geq 1$, in a
neighborhood $U$ of the closed unit disk $\ov{D^{2n}}\subset
\bc^n, \, n \geq 3$. Assume that the singularity of $\Omega$
inside the disk is the origin. Let $\Sigma$ be the set of contact
points between spheres $S^{2n-1}(r), \, r >0$ and the foliation
$\fa(\Omega)$ defined by $\Omega$. We denote by $\vec R =
\sum\limits_{j=1}^n z_j \frac{\partial}{\partial z_j}$ the complex
radial vector field. Then $\Sigma$ is a non-empty $\vec
R$-invariant set.
\end{Proposition}

\begin{proof}
The assumption on the singular set of $\Omega$ implies that
\linebreak $\cod _\bc (\Sing(\Omega))_0 =n\geq 3$. According to
Malgrange's theorem \cite{[Malgrange]}, there exists a holomorphic
first integral $f\colon V \to \bc$ of $\Omega$ in a neighborhood
$V$ of the origin; $\Omega\big|_V= gdf$, where $g$ is a non-zero
holomorphic function on $V$. Since a point of
$S^{2n-1}(\epsilon)\subset V$ where a maximal value of
$|f|\big|_{S^{2n-1}(\epsilon)}$ is attained is a contact point,
$\Sigma\big|_V$ is non-empty. Take a point $p=(p_1,...,p_n)\in
\Sigma$, the orbit of $\vec R$ passing through $p$ is $p. e^T=(p_1
. e^T,...,p_n. e^T), \, T \in \bc$. We will prove that $\ov{f_j(p.
e^T)} = \lambda(T). p_j.e^T, \, \lambda(T)\in \bc, \, 1 \leq j
\leq n$. Since $p \in \Sigma$ we have $\ov{f_j(p)}=\lambda. p_j,
\lambda \in \bc, \, 1 \leq j \leq n$. For any $j, \, \ov{f_j(p .
e^T)}=(\lambda. \ov{e^{kT}}. e^{-T}).p_j e^T$. Put
$\lambda(T)=\lambda. \ov{e^{kT}}. e^{-T}$, then $p. e^T \in
\Sigma$.

\end{proof}

\section{The projected gradient vector field}
\label{section:projectedgradient} Let $\Omega$ be an integrable
holomorphic one-form with $\sing(\Omega)=\{0\}$ and $\vr$ be the
distance function as above. Under the same notation of
\S\ref{section:varietyofcontacts}, we define the projection
$t_{\fa(\Omega)}(\vec n)$ of $\vec n$ onto $T(\fa(\Omega))$ along
$\grad(\Omega)$ as:
\begin{equation}
t_{\fa(\Omega)}(\vec n)=\vec n - \big[ <\vec n,
\frac{X}{||X||}>\frac{X}{||X||} + <\vec n, \frac{JX}{||JX||}>
\frac{JX}{||JX||} \big ]
\end{equation}
where $<,>$ is the natural inner product of $\re^{2n}\simeq \bc
^n$ and $||X||$ is the norm of $X$.

In the case of  a holomorphic vector field $Z$ in $\bc^n, n \geq
2$, X. Gomez-Mont, J. Seade and A. Verjovsky (\cite{GSV}) defined
the real analytic vector field $r_Z= -1\ov{<Z, \vec R>_\bc}. Z \in
Z $ and investigated the properties of $r_Z$, where $<Z, \vec
R>_\bc$ is the canonical hermitian inner product of $\bc^n$
between $Z$ and $\vec R$. Our definition is a version for
holomorphic foliations of codimension one. In the following, we
explain some properties of $t_{\fa(\Omega)}(\vec n)$.

\begin{Proposition}
\label{Proposition:propertiesofprojectedvector} We have:
\begin{enumerate}

\item $t_{\fa(\Omega)}(\vec n)\in T(\fa(\Omega)).$

\item $\Sing(t_{\fa(\Omega)}(\vec n))=\Sigma.$

\item $\grad(\vr \big|_L)= 2 t_{\fa(\Omega)}(\vec n).$

 \item Away from the variety of contacts $\Sigma$, the
distributions $\fa(\Omega)$ and $\fa(d\vr)$ meet transversally.

\item Let $L$ be an integral submanifold  of $\fa(\Omega)$. The
vector field $t_{\fa(\Omega)}(\vec n)$ is transversal to each
level surface of $\vr\big|_{L}$ on $L\setminus (L\cap \Sigma)$.
\end{enumerate}

\end{Proposition}
\begin{proof}

\noindent (1) \,   $<t_{\fa(\Omega)}(\vec n), X>=<\vec n, X> -
\{<\vec n, \frac{X}{||X||}><\frac{X}{||X||}, X> + <\vec n,
\frac{JX}{||JX||}><\frac{JX}{||JX||}, X>\}=0 $ Similarly,
$<t_{\fa(\Omega)}(\vec n), JX>=0$.

\noindent (2) Given $p \in \Sing(t_{\fa(\Omega)}(\vec n))$ we have
$ \vec n \in \grad(\Omega)$ and then $p \in \Sigma$. By its turn,
$p \in \Sigma$ implies $\vec n = a X + b JX, a, b, \in \mathbb R$
and then
\begin{equation}
\begin{cases}
<\vec n, \frac{X}{||X||}>\frac{X}{||X||} = a X \\
<\vec n, \frac{JX}{||JX||}>\frac{JX}{||JX||}=bJX
\end{cases}
\end{equation}
This finally implies $t _{\fa(\Omega)} (\vec n)=0$.

(3) Let $p$ be a point of $\ov{D^{2n}}\setminus \{0\}$. Since
$\Omega$ is non-singular at $p$, there exists a neighborhood $U$
of $p$ such that $\Omega\big|_U$ writes $\Omega\big|_U= G dF$ on
$U$ where $G$ and $F$ are holomorphic functions on $U$ and $G$ is
a non-zero holomorphic function on $U$. The leaf $L$ through $p$
is given by $L=\{ z \in U \big| F(z)=F(p)=c\}$. We can assume that
$f_n(z)\ne 0$ on $U$ because $\Omega$ is non-singular on $U$. By
the implicit function theorem, we can take a local coordinates
$(z_1,...,z_{n-1}, z_n(z_1,...,z_{n-1}, c))$ to represent the leaf
$L$. Therefore, we have the relations:
\begin{equation}
f_j(z) + f_n (z) \frac{\partial z_n}{\partial z_j} =0, \, 1 \leq j
\leq n-1,
\end{equation}

i.e.,
\begin{equation}
\begin{cases} g_j + g_n \frac{\partial x_n}{\partial x_j} - h_n
\frac{\partial y_n}{\partial x_j}=0 \\
h_j + h_n \frac{\partial x_n}{\partial x_j} + g_n \frac{\partial
y_n}{\partial x_j}=0, \, 1 \leq j \leq n-1
\end{cases}
\end{equation}

Then we have the gradient $\grad(\vr \big|_L)$ of
$\vr\big|_L(z)=\sum\limits_{j=1}^{n-1} (x_j ^2 + y_j ^2) \, + (x_n
^2 + y_n ^2)$ and then

\begin{equation}
\grad(\vr\big|_L) = 2\{\sum\limits_{j=1}^{n-1} ( x_j + x_n
\frac{\partial x_n}{\partial x_j} + y_n \frac{\partial
y_n}{\partial x_j})\frac{\partial}{\partial x_j} +
\sum\limits_{j=1}^{n-1} ( y_j + x_n \frac{\partial x_n}{\partial
y_j} + y_n \frac{\partial y_n}{\partial y_j})
\frac{\partial}{\partial y_j}\}
\end{equation}

On the other hand, directly calculating, we get
$(t_{\fa(\Omega)}(\vec n))= \frac{1}{2} \grad( d \vr\big|_L):$

\[
(t_{\fa(\Omega)}(\vec n)) = \vec n - \frac{\lg \vec n, X
\rg}{||X||^2} [ \sum\limits_{j=1}^{n-1} ( g_j
\frac{\partial}{\partial x_j} - h_j \frac{\partial}{\partial
y_j})] - \frac{\lg \vec n, JX \rg}{||JX||^2}
\big[\sum\limits_{j=1}^{n-1}( h_j \frac{\partial}{\partial x_j} +
g_j \frac{\partial}{\partial x_j})\big]
\]
\[
- \frac{\lg \vec n, X \rg}{||X||^2} \{ g_n [
\sum\limits_{j=1}^{n-1} (\frac{\partial x_n}{\partial x_j}
\frac{\partial}{\partial x_j} + \frac{\partial x_n}{\partial y_j}
\frac{\partial}{\partial y_j})]  - h_n[ \sum\limits_{j=1}^{n-1}
(\frac{\partial y_n}{\partial x_j}\frac{\partial}{\partial x_j} +
\frac{\partial y_n}{\partial y_j}\frac{\partial}{\partial y_j})]\}
\]
\[
 - \frac{\lg \vec n, JX
\rg}{||JX||^2}\{ h_n [ \sum\limits_{j=1}^{n-1}(\frac{\partial
x_n}{\partial x_j} \frac{\partial}{\partial x_j} + \frac{\partial
x_n}{\partial y_j}\frac{\partial}{\partial y_j}) ] + g_n[
\sum\limits_{j=1}^{n-1} ( \frac{\partial y_n}{\partial
x_j}\frac{\partial}{\partial x_j} + \frac{\partial y_n}{\partial
y_j} \frac{\partial}{\partial y_j})] \}
\]
\[
 =\sum\limits_{j=1}^{n-1} (x_j + x_n \frac{\partial x_n}{\partial
 x_j} + y_n \frac{\partial y_n}{\partial x_j})
 \frac{\partial}{\partial x_j} + \sum\limits_{j=1}^{n-1}( y_j +
 x_n \frac{\partial x_n}{\partial y_j} + y_n \frac{\partial
 y_n}{\partial y_j})\frac{\partial}{\partial y_j} = \frac{1}{2}
 \grad(\vr\big|_L).
 \]

Item (4)  follows immediately from (2). (5) follows from (3).

\end{proof}

\subsection{The complex projected tangential vector field}
\label{subsection:complexprojectedtangential} The {\it complex
projected tangential vector field} $\mathcal T _ {\fa(\Omega)}
(\vec R) \in T (\fa(\Omega))$ is defined as
\begin{equation}
\mathcal T _ {\fa(\Omega)} (\vec R):= \vec R - \mu .\grad
(\Omega)= \frac{1}{2} [ t _ {\fa (\Omega)} ( \vec n ) - \sqrt{-1}
J  t_ {\fa (\Omega)} (\vec n)]
\end{equation}
where
\begin{equation}
\mu = \frac{ \lg \vec R, \grad (\Omega) \rg } { || \grad(\Omega)||
^2}
\end{equation}

Then we have

\begin{equation}
\Sigma = \sing(\mathcal T _{\fa (\Omega)} (\vec R)) = \{ z \in \bc
^n : \, \vec R = \mu . \grad(\Omega)\}
\end{equation}

\section{The variety of contacts of a Morse foliation}
\label{section:varietycontactsMorse}

Let $\Omega$ be an integrable holomorphic one-form with
$\Sing(\Omega)=\{0\}$. Under the same notation of
\S~\ref{section:varietyofcontacts}, we give a characterization of
Morse type foliations.

\begin{Proposition}
\label{Proposition:(1)implies(2)} Assume that each critical point
$p \in L\cap \Sigma$ of $\vr\big|_L$ is non-degenerate. Then the
real dimension of $\Sigma\setminus \{0\}$ is two and
$\Sigma\setminus \{0\}$ is transverse to the foliation
$\fa(\Omega)$.
\end{Proposition}
\begin{proof}
 First we note the existence of a
critical point of $\vr\big|_{L^\prime}$ for any leaf $L^\prime$
close enough to $L$ at $p$, i.e., passing through a neighborhood
of $p\in L \cap \Sigma$, where $p$ is a critical point of
$\vr\big|_L$, is proved by Poincar\'e-Hopf theorem (or Pugh's
generalization of Poincar\'e-Hopf theorem). Take a distinguished
coordinate neighborhood $(w_1,...,w_{n-1},w_n)\in U$ for
$\fa(\Omega)$ where $p\in U$ corresponds to the origin and
$L^\prime \cap U=\{w_n=c\}$. Then $\Sigma \cap U$ is defined by
the equations:
\[
\frac{\partial \vr}{\partial w_1}=0,...,\frac{\partial
\vr}{\partial w_{n-1}}=0
\]
i.e., $\frac{\partial \vr}{\partial u_j}=0, \, j=1,...,2n-2$,
where we write $w_j=u_{2j-1} + \sqrt{-1} u_{2j}, \, 1 \leq j \leq
n$, by the real coordinate. Since $p$ is non-degenerate on $L$,
\[
\det(\frac{\partial ^2 \vr}{\partial u_i \partial u_j})_{1 \leq i,
j \leq 2n-2}
\]
 is different from $0$. By the Implicit function theorem, $\Sigma$
 is parametrized by $u_{2n-1}$ and $u_{2n}: \,
 (u_1(u_{2n-1},u_{2n}),...,u_{2n-2}(u_{2n-1},u_{2n}), u_{2n-1},
 u_{2n})$. Then the dimension of $\Sigma\setminus \{0\}$ is two
 and $\Sigma-\{0\}$ is
 transverse to $\fa(\Omega)$.

\end{proof}

\begin{Proposition}
\label{Proposition:(2)implies(1)} If the real dimension of
$\Sigma\setminus \{0\}$ is two and $\Sigma \setminus \{0\}$ is
transverse to the foliation $\fa(\Omega)$, then each critical
point $p \in \Sigma\cap L$ of $\vr\big|_L$ on the leaf $L$ passing
through $p$ is non-degenerate.

\end{Proposition}

\begin{proof}
Take a distinguished coordinate neighborhood $U$,
$(w_1,...,w_{n-1}, w_n)$, at $p$ such that the leaf $L \cap U$ of
$\fa(\Omega)$ passing through $p$ is defined by $\{w_n=0\}$. Using
the real coordinate $(u_1,...,u_{2n-2}, u_{2n-1}, u_{2n}): \, w_j
= u_{2j-1} + \sqrt{-1} u_{2j}, \, 1, \leq j \leq n, \, \Sigma$ is
parametrized in $U$ by $u_{2n-1}$ and $u_{2n}$:
\[
\Sigma \cap U = \{(u_1(u_{2n-1}, u_{2n}), ..., u_{2n-2}(u_{2n-1},
u_{2n}),u_{2n-1}, u_{2n}) \big| u_{2n-1}, u_{2n} \}. \]

 Then the tangent space of $T\Sigma\big|_U$ of $\Sigma \cap U$ is
 generated by $\vec v_{2n-1}$ and $ \vec v_{2n}$:

 \[
 \vec v_{2n-1} = \sum\limits_{j=1}^{2n-2} \frac{\partial
 u_j}{\partial u_{2j-1}} \frac{\partial}{\partial u_j} +
 \frac{\partial}{\partial u_{2n-1}}
 \]
 \[
 \vec v_{2n} = \sum\limits_{j=1}^{2n-2} \frac{\partial
 u_j}{\partial u_{2n}}\frac{\partial}{\partial u_j} +
 \frac{\partial}{\partial u_{2n}}
 \]

On the other hand, $\Sigma$ is defined by the critical points of
the distance function $\vr$ as follows:
\[
\Sigma\cap U \{ w \in U \big| \frac{\partial \vr}{\partial
u_1}=0,...,\frac{\partial \vr}{\partial u_{2n-2}}=0\}
\]

Then the tangent space $T \Sigma\big|U$ is defined by $\{
d(\frac{\partial \vr}{\partial u_1})=0,...,d(\frac{\partial
\vr}{\partial u_{2n-2}})=0\}: $
\[
T\Sigma\big|_U = \{ \vec v \in T\mathbb R ^{2n} \big| d (
\frac{\partial \vr}{\partial u_i})(\vec v)=0, \, 1 \leq i \leq
2n-2\}
\]

Since $d(\frac{\partial \vr}{\partial u_i})(\vec v _k)=0$ for
$k=2n-1, 2n, \, 1 \leq i \leq 2n -2, $ we get the following
relations:
\[
\begin{pmatrix}
\frac{\partial ^2 \vr}{\partial u_1 \partial u_1} \ldots
\frac{\partial ^2 \vr}{\partial u_{2n-2} \partial u_1} \\
\hdots \ldots \hdots\\
 \frac{\partial ^2 \vr}{\partial u_1
\partial u_n} \ldots \frac{\partial ^2 \vr}{\partial u_{2n-2}
\partial u_{2n-2}}
\end{pmatrix}
\begin{pmatrix}
\frac{\partial u_1}{\partial u_k} \\
\hdots \\
\frac{\partial u_{2n-2}}{\partial u_k}
\end{pmatrix} +
\begin{pmatrix}
\frac{\partial ^2 \vr}{\partial u_k \partial u_1}\\ \hdots
\\ \frac{\partial ^2 \vr}{\partial u_{k} \partial u_{2n-2}}
\end{pmatrix}
=
\begin{pmatrix}
0 \\ \hdots \\
0
\end{pmatrix}
\]

Then the fact that the rank of $\{ d(\frac{\partial \vr}{\partial
u_1}),...,d(\frac{\partial \vr}{\partial u_{2n-2}})\}$ is $2n-2$
means that the rank of the $(2n-2)\times (2n-2)$ matrix
\[
\begin{pmatrix}
\frac{\partial ^2 \vr}{\partial u_1 \partial u_1} \ldots
\frac{\partial ^2 \vr}{\partial u_{2n-2} \partial u_1} \\
\hdots \ldots \hdots\\
 \frac{\partial ^2 \vr}{\partial u_1
\partial u_n} \ldots \frac{\partial ^2 \vr}{\partial u_{2n-2}
\partial u_{2n-2}}
\end{pmatrix}
\]
is $2n-2$. Therefore the critical point $p $ is non-degenerate.

\end{proof}

Summarizing
Proposition~\ref{Proposition:propertiesofprojectedvector} (3) and
Propositions~\ref{Proposition:(1)implies(2)} and
\ref{Proposition:(2)implies(1)} we have the following theorem:

\begin{Theorem}[Characterization of Morse type foliations]
\label{Theorem:caracterizationMorsetype} Given a holomorphic
integrable one-form $\Omega$ with $\Sing(\Omega)=\{0\}$,  in a
neighborhood of the disk $\ov{D^{2n}}$, the following conditions
are equivalent:

\begin{itemize}
\item[{\rm(i)}] $\fa(\Omega)$ is of Morse type, i.e., each
critical point $p \in \Sigma \cap L$ of $\vr\big|_L$ on each leaf
$L$ is nondegenerate.

\item[{\rm(ii)}]  $\Sigma - \{0\}$ has real dimension two and is
transverse to the foliation $\fa(\Omega)$.

\item[{\rm(iii)}] The singularities of $t _ {\fa (\Omega)} ( \vec
n )$ on each leaf are nondegenerate.

\end{itemize}

\end{Theorem}

\begin{Corollary}
If the real dimension of $\Sigma\setminus \{0\}$ is two and
$\Sigma\setminus \{0\}$ is transverse to the foliation
$\fa(\Omega)$, then $\Sigma$ has a cone structure.

\end{Corollary}

\begin{proof}
We denote by $t_\Sigma(\vec n)$ the projection of $\vec n$ onto
$T\Sigma$ along $T(\fa(\Omega))$. We note that $t_\Sigma(\vec n)$
is different from zero. Hence, the orbits of $t_\Sigma(\vec n)$
define a cone structure for $\Sigma$.

\end{proof}

\section{Proof of the non-existence theorem}
\label{section:proofofmain}

Let $\Omega$ be an integrable holomorphic one-form  of Morse type
and $\Sigma(\Omega)$ of dimension two. Let  $L\in \fa(\Omega)$ be
a leaf of $\fa(\Omega)$ such that $\Sigma\cap L \ne \emptyset$,
indeed we assume that there is a point $p \in L$ such that
$\vr\big|_L$ has a nondegenerate critical point of index $0$ at
$p$ (in other words, $p$ is a local minimum point for
$\vr\big|_L$).
\begin{Lemma}
\label{Lemma:indexzero} There is a neighborhood $U$ of $p$ in
$\ov{D^{2n}}$ such given a leaf $L^\prime$ intersecting $U$, the
restriction $\vr\big|_{L ^\prime}$ exhibits a critical point
$p^\prime \in U$ of index $0$.
\end{Lemma}

\begin{proof}
Indeed, we choose a distinguished coordinate neighborhood $U$ for
$\fa(\Omega)$, such that $\fa(\Omega)\big|_U$ is equivalent to a
foliation by $(2n-2)$-discs. The variety of contacts $\Sigma$ has
a cone structure and is  a $2$-dimensional analytic manifold in a
neighborhood of $p$  transverse to $\fa(\Omega)$. The level
surfaces of $\vr\big|_{L\cap U}$ are $(2n-3)$-dimensional spheres
on $L$ centered at $p$. Moreover, for small $\epsilon>0$ a
connected component $\Sigma_0$ of the intersection $\Sigma \cap
S^{2n-1}(\epsilon)$ is a circle $S^1$. Hence, $\Sigma_0$ is a cone
over $S^1$ and the gradient vector field $t_{\fa(\Omega)}(\vec n)$
of $\vr\big|_L$ is tangent to $L$ and has a Morse type singularity
at $p$, of index zero\footnote{We can therefore consider the
situation as $\Sigma\cap U$ being a (cylinder) product of the
intersection a  circle $\gamma \subset (\Sigma\cap S^{2n-1}(r)\cap
U)$ by an interval $(-\delta, \delta)$, with $p$ corresponding to
the level zero. The  level surfaces of $\vr\big|_{L\cap U}$ are
$(2n-3)$-spheres centered at $p$ and therefore level surfaces of
$\vr\big|_U$ can be thought as cylinders transverse to
$\fa(\Omega)$ off $\Sigma\cap U$.}. Given a leaf $L^\prime$ such
that $L^\prime$ intersects $U$ at a point close enough to $p$,
then by Poincar\'e-Hopf theorem, applied to a disc
$D^\prime\subset L^\prime \cap U$, obtained as lifting of a
sufficiently small disc $D\subset L\cap U$ centered at $p$ whose
boundary is a level curve of $\vr\big|_L$, we conclude that
$t_{\fa(\Omega)}(\vec n)\big|_{L^\prime}$ on $L^\prime$ has a
Morse type singularity at some point $p^\prime \in D^\prime$. This
singularity of Morse index zero and corresponding to a point of
minimum for the restriction $\vr\big|_{L^\prime}$.
\end{proof}

Let now $\Sigma$ be of dimension two and transverse to
$\fa(\Omega)$. Let $L\in \fa(\Omega)$ be a leaf of $\fa(\Omega)$
such that $\Sigma \cap L\ne \emptyset$, indeed we assume that
there is a point $p\in L $ such that $\vr\big|_L$ has a
non-degenerate critical point  of index $k>0$ at $p$.

\begin{Lemma}
\label{Lemma:indexk} There is a neighborhood $U$ of $p$ in
$\ov{D^{2n}}$ such that given a leaf $L^\prime$ intersecting $U$,
the distance function $\vr\big|_{L^\prime}$ on $L^\prime$ exhibits
a critical point $p^\prime \in U \cap L^\prime$ of index $k$.
\end{Lemma}

\begin{proof}
The proof is similar to the above. Choose a distinguished
coordinate neighborhood $U$ for $\fa(\Omega)$, such that
$\fa(\Omega)\big|_U$ is equivalent to a foliation by
$(2n-2)$-discs. By Morse's lemma, we have a sufficiently small
disc $D\subset L \cap U$ centered at $p$ whose boundary is in
generic position for $t_{\fa(\Omega)}(\vec n)\big|_L$. Arguing as
in the proof of Lemma~\ref{Lemma:indexzero} we obtain a disc
$D^\prime \subset L^\prime \cap U$ as lift of $D\subset L \cap U$,
and we may deform $D^\prime$ to a disc $\widetilde{D^\prime}$
whose boundary is in generic position for $t_{\fa(\Omega)}(\vec
n)\big|_{L^\prime}$. By Pugh's generalization of Poincar\'e-Hopf
theorem we conclude that $t_{\fa(\Omega)}(\vec n)\big|_{L^\prime}$
has a Morse type singularity at some point $p^\prime \in
\widetilde{D^\prime}$, of Morse index $k$.
\end{proof}

\begin{proof}[Proof of Theorem~\ref{Theorem:main}]

Assume by contradiction that $\fa$ is transverse to the boundary
sphere $S^{2n-1}(1)$. We note that $\fa$ is defined by an
integrable holomorphic one-form $\Omega$ in $U\supset \ov{D^{2n}}$
and according to \cite{Ito-Scarduatopology} there is a unique
singular point  $p$ of $\Omega$ inside the disc $\ov{D^{2n}}$ and
this is a non-degenerate singular point: the determinant of the
matrix $D(\Omega)(p)$ of the coefficients of the linear part of
$\Omega$ is different from $0$. By a M\"obius transformation we
can assume that the origin is this singular point. Since $\Omega$
is integrable and  $n\geq 3$,  by Malgrange's theorem
(\cite{[Malgrange]}) $\Omega$ admits a local holomorphic first
integral $f\colon V \to \bc$, \, $\Omega\big|_V= gdf$ in a
neighborhood $V$ of $0$ in $\bc^n$, where $g$ is a non-zero
holomorphic function in $V$. Then $\fa(\Omega)\big|_V$ is defined
by the level hypersurfaces of $f$ and only finitely many leaves
accumulate on the origin. Let $\vr$ be the distance function with
respect to the origin $0\in \bc^n$. We consider the variety of
contacts $\Sigma=\Sigma(\Omega, \vr)$. Since by hypothesis $\fa$
is of Morse type there is a minimal point (that is, a critical
point of Morse index $0$) for the restricted distance function
$\vr\big|_L$  to any leaf $L$ intersecting $V$ and these points
belong to $\Sigma$. Let $\Sigma _0 ^*$ be a  connected component
of $\Sigma \setminus \{0\}$ which contains such minimal points in
a neighborhood of the origin. According to
Theorem~\ref{Theorem:caracterizationMorsetype} and
Lemma~\ref{Lemma:indexzero}, $\Sigma_0 ^*$ in $U \supset
\ov{D^{2n}}$ is of real dimension two and has an unbounded cone
structure. Then $\Sigma_ 0 ^*$ intersects $S^{2n-1}(1)$. This is
in contradiction to the assumption of transversality.

\end{proof}

\begin{Remark}
{\rm Thought Lemma~\ref{Lemma:indexk} has not been used in a
direct way in the proof of Theorem~\ref{Theorem:main} (it is
enough to begin with a minimal (index zero) singular point for
$\vr\big|_L$ for a leaf $L$ close to the origin) we have stated it
because we think it will be useful in a more general setting.}
\end{Remark}

\section{Examples}
\label{section:examples}

\begin{Example}
{\rm Consider the function $f(z)=\sum\limits_{j=1}^n z_j ^2$. The
foliation $\fa(df)$ is defined by the level surfaces of $f$. The
variety of contacts $\Sigma$ between $\fa(df)$ and $\fa(d\vr)$ is
defined by the equations: $ 2 \ov z_j = \lambda z_j, 1, \leq j
\leq n, \lambda \in \bc.$

We will write complex numbers $z_j = r_j e^{\sqrt{-1}\theta _j}$
and $c=re^{\sqrt{-1}\theta}$ in polar coordinates.

We will denote by $L_c$ the leaf of $\fa(df)$ defined by
$\{f=c\}$. To explain $\Sigma$ we prepare some notations:

\begin{equation}
\Sigma_i ^1 \cap L_c=\{(0,...,0,r_i e^{\sqrt{-1} \theta_i},
o,...,0)| r_i ^2 = r, \theta_i = \frac{\theta}{2} {\rm or }
\frac{\theta}{2} + \pi \}
\end{equation}

\begin{equation}
\Sigma_{i,j} ^2 \cap L_c=\{(0,...,0,r_i e^{\sqrt{-1} \theta_i},
0,...,r_j e^{\sqrt{-1} \theta _j},0,...,0) | r_i ^2  + r_j ^2= r,
\theta_i, \theta_j = \frac{\theta}{2} {\rm or } \frac{\theta}{2} +
\pi \}
\end{equation}

\begin{equation}
\Sigma^n_{1,2,...,n} \cap L_c = \{(r_1 e^{\sqrt{-1}
\theta_1},...,,r_n e^{\sqrt{-1} \theta_n})| r_1 ^2+...+ r_n ^2 =
r, \theta_i = \frac{\theta}{2} {\rm or } \frac{\theta}{2} + \pi\}
\end{equation}

Each dimension of $\Sigma_{i_1,...,i_\ell} ^\ell \cap L_c$ is
$\ell-1$ and $\dim(\Sigma ^\ell _{i_1,...,i_{\ell}})$ is equal to
$\ell + 1$. Then we can explain $\Sigma$ as follows:

\begin{equation}
\Sigma=(\bigcup _{i=1}^n \Sigma_ i ^1) \cup(\bigcup \Sigma_{i,j}
^2) \cup ... \cup (\Sigma_{1,2,...,n} ^n).
\end{equation}

In this degree two homogeneous case $\Sigma$ is a $\vec
R$-invariant set.

}
\end{Example}

\begin{Example}
{\rm Let $\lambda_j \, \, (1 \leq j \leq n)$ be complex numbers
such that $|\lambda_i|\ne |\lambda_j|$ for $i \ne j$. Consider the
function $f_\lambda(z)=\sum\limits_{j=1}^n \lambda_j z_j ^2$. We
define the foliation $\fa(df_\lambda)$. The variety of contacts
$\Sigma$ between $\fa(df_\lambda)$ and $\fa(d\vr)$ is defined by
the equations:
\begin{equation}
2 \ov \lambda_j \ov z_j = \xi z_j, \, 1 \leq j \leq n, \, \xi \in
\bc.
\end{equation}

$\Sigma\setminus \{0\}$ consists of $n$-connected components
$\Sigma_1 \cup \Sigma_ 2 \cup ... \cup \Sigma_n$ where
$\Sigma_i=\{0,...,0, r_i e^{ \sqrt{-1} \theta_i},0,...,0) | r_i
>0, 0\leq \theta _i < 2\pi\}$.

Let $L_c=\{f_\lambda = c\}$ be a leaf of $\fa(d f_\lambda)$. Take
$p=(\omega_1,0,...,0)\in \Sigma_1 \cap L_c$. From now on, let us
calculate the real hessian matrix of $\vr\big|_{L_c}$ represented
by the equations $(z_1(z_2,...,z_n,c), z_2,...,z_n), \, z_i = x_i
+ \sqrt{-1} y_i$, in a neighborhood of $p \in \Sigma_1 \cap L_c$.
First we note that $(z_1(z_2,...,z_n,c), z_2,...,z_n)$ is the
holomorphic function of complex variables $z_2,...,z_n$. By
Cauchy-Riemann equations, we get
\begin{equation}
\frac{\partial z_1}{\partial z_j} = \frac{\partial x_1}{\partial
x_j} + \sqrt{-1} \frac{\partial y_1}{\partial x_j}= \frac{\partial
y_1}{\partial y_j}- \sqrt{-1}\frac{\partial x_1}{\partial y_j}.
\end{equation}

Differentiating $c=f_\lambda(z)$ by $\frac{\partial}{ \partial
z_j}$, we have
\begin{equation}
0=2\lambda_1 z_1 \frac{\partial z_1}{\partial z_j} + 2 \lambda_j
z_j.
\end{equation}
Since $0=2\lambda_1 \omega_1(\frac{\partial x_1}{\partial x_j} (p)
+ \sqrt{-1}\frac{\partial y_1}{\partial x_j} (p))$, we have
$\frac{\partial x_1}{\partial x_j}(p)=\frac{\partial y_1}{\partial
y_j}(p)=0$ and $\frac{\partial y_1}{\partial x_j}(p)= -
\frac{\partial x_1}{\partial y_j}(p)=0$. Then we obtain
\begin{equation}
\frac{\partial \vr}{\partial x_j}(p)=2u_1 \frac{\partial
x_1}{\partial x_j}(p) + 2 v_1 \frac{\partial y_1}{\partial x_j}(p)
+ 2.0 = 0
\end{equation}
and
\begin{equation}
\frac{\partial \vr}{\partial y_j}(p) = 2u_1 \frac{\partial
x_1}{\partial y_j}(p) + 2 v_1 \frac{\partial y_1}{\partial y_j}(p)
+ 2.0=0
\end{equation}
where $w_1=u_1 + \sqrt{-1} v_1$.

To obtain the real Hessian matrix $H(\vr\big|_{L_c})(p)$ at $p$ we
calculate the following equations:

\begin{equation}
\frac{\partial ^2 \vr}{\partial x_i \partial x_j}= 2
\frac{\partial x_1}{\partial x_i} \frac{\partial x_1}{\partial
x_j} + 2x_1 \frac{\partial ^2 x_1}{\partial x_i \partial x_j} + 2
\frac{\partial y_1}{\partial x_i}\frac{\partial y_1}{\partial x_j}
+ 2 y_1\frac{\partial ^2 y_1}{\partial x_i \partial  x_j} + 2
\delta_{ij}
\end{equation}

\begin{equation}
\frac{\partial ^2 \vr}{\partial y_i \partial x_j}= 2
\frac{\partial x_1}{\partial y_i} \frac{\partial x_1}{\partial
x_j} + 2x_1 \frac{\partial ^2 x_1}{\partial y_i \partial x_j} + 2
\frac{\partial y_1}{\partial y_i}\frac{\partial y_1}{\partial x_j}
+ 2 y_1\frac{\partial ^2 y_1}{\partial y_i \partial  x_j}
\end{equation}

\begin{equation}
\frac{\partial ^2 \vr}{\partial y_i \partial y_j}= 2
\frac{\partial x_1}{\partial y_i} \frac{\partial x_1}{\partial
y_j} + 2x_1 \frac{\partial ^2 x_1}{\partial y_i \partial y_j} + 2
\frac{\partial y_1}{\partial y_i}\frac{\partial y_1}{\partial y_j}
+ 2 y_1\frac{\partial ^2 y_1}{\partial y_i \partial  y_j} + 2
\delta_{ij}
\end{equation}

Differentiating $0= 2 \lambda_1 z_1 \frac{\partial z_1}{\partial
z_j}+ 2 \lambda_j z_j$ by $\frac{\partial}{\partial z_i}$, we have
\begin{equation}
\frac{\partial ^2 z_1}{\partial z_i \partial z_j}(p) =
\begin{cases}
-\frac{\lambda_i}{\lambda_1}. \frac{1}{w_1}, \, (i \ne j) \\
0 \, \, (i \ne j)
\end{cases}
\end{equation}

}
\end{Example}

Put $\frac{\lambda_i}{\lambda_1}=a_i + \sqrt{-1} b_i$.  Then we
get each component of $H(\vr \big|_{L_c})(p)$:

\begin{equation}
\frac{\partial ^2 \vr}{\partial x_i \partial x_j}(p) =
\begin{cases} 2[- \frac{ a_i (u_i^2 - v_i ^2) + 2b_i u_1
v_1}{|w_1|^2} + 1] , \, \, (i = j) \\
0 \, , (i\ne j)
\end{cases}
\end{equation}
\begin{equation}
\frac{\partial ^2 \vr}{\partial y_i \partial x_j}(p) =
\begin{cases} 2[ \frac{ b_i (u_i^2 - v_i ^2) - a_i u_1
v_1}{|w_1|^2}] , \, \, (i = j) \\
0 \, , (i\ne j)
\end{cases}
\end{equation}
\begin{equation}
\frac{\partial ^2 \vr}{\partial y_i \partial y_j}(p) =
\begin{cases} 2[ \frac{ a_i (u_i^2 - v_i ^2) + 2b_i u_1
v_1}{|w_1|^2} + 1] , \, \, (i = j) \\
0 \, , (i\ne j)
\end{cases}
\end{equation}

Then the characteristic equation of $H(\big|_{L_c})(p)$ is
\begin{equation}
\Pi _{i=2}^n [ (1- \mu)^2  - (a_i ^2 + b_i ^2)]=0
\end{equation}
that is, $\mu = 1 \pm \sqrt{ a_i ^2 + b _i ^2}, \, ( 2 \leq i \leq
n)$. The sign  of the eigenvalues $\mu$ is defined by
$\frac{|\lambda_i|}{\lambda_1|} > 1$ or
$\frac{|\lambda_i|}{|\lambda_1|} <1$. Furthermore, we assume that
$|\lambda_1|> |\lambda_2|> ...>|\lambda_n|$, each $\Sigma_i$
consists of critical points of Morse index $i-1$ of
$\vr\big|_{L_c}$.

\begin{Example}
{\rm  Let $\epsilon_j \, \, (1 \leq j \leq n)$ be sufficiently
small complex numbers such that $|1 + \epsilon_i| \ne | 1 +
\epsilon _j|$ for $ i \ne j$. Consider $f_\epsilon (z)=
\sum\limits_{j=1}^n (1+ \epsilon _j) z_j ^2$. We mean that
$f_\epsilon(z)$ is a deformation of $f(z)=\sum\limits_{j=1}^n z_j
^2$. We can check that the dimension of the variety of contacts
$\Sigma(df)$ is $n+1$ and the dimension of $\Sigma(df_\epsilon)$
is two.}
\end{Example}


\section{Linear foliations of Morse type}
\label{section:linearfoliationsMorse}
 Let
$\Omega=\sum\limits_{i=1}^ n \alpha_i(z) dz_i$ be a linear
one-form on $\bc^n, \, n \geq 3$, defined by a quadratical
polynomial $f(z)=\frac{1}{2}\sum\limits_{i,  j =1} ^n a_{ij} z_i
z_j$, with $a_{ji}=a_{ij}$: \, $\Omega=df$. We denote by
$A=(a_{ij})$ the $n \times n$ complex invertible matrix with
$(i,j)$ component $a_{ij}$. The corresponding gradient vector
field is $\grad(\Omega)=\sum\limits_{i=1}^n
\ov{\alpha_i(z)}\frac{\partial}{\partial z_i}$, and the radial
vector field of the distance function $\vr(z)=\sum\limits_{j=1}^n
|z_i|^2 $ in $\bc^n$ is $\vec R=\sum\limits_{i=1}^n
z_i\frac{\partial}{\partial z_i}$. The corresponding variety of
contacts $\Sigma$ between $\fa(\Omega)$ and $\{d \vr=0\}$ is
defined by the following equation:
\begin{equation}
\Sigma=\{z \in \bc^n : \, \vec R(z)= \mu(z). \grad(\Omega)(z)\}
\end{equation}
where
\begin{equation}
\mu(z)=\frac{\lg \vec R(z), \grad(\Omega)(z)
\rg}{||\grad(\Omega)(z)||^2}
\end{equation}
By its turn  $\vec R(z)= \mu(z). \grad(\Omega)(z)$ is equivalent
to
\begin{equation}
z_j=\mu(z).\ov{\alpha_j(z)}, \, \forall j=1,...,n \end{equation}

i.e.,

\begin{equation}
\label{equation:linear}  A\begin{pmatrix} z_1 \\ \vdots \\
z_n\end{pmatrix} = \frac{1}{\ov {\mu(z)}}
\begin{pmatrix} \ov {z_1} \\ \vdots \\ \ov {z_n} \end{pmatrix}
\end{equation}
Since $A$ is nonsingular we can write
\begin{equation}
z = \frac{1}{\ov {\mu(z)}} A^{-1} (\ov z) = \frac{1}{\ov {\mu(z)}}
A^{-1} \ov {\bigg (\frac{1}{\ov{{\mu(z)}}} A^{-1} \ov z\bigg)} =
\frac{1}{|{\mu(z)}|^2} A ^{-1} . \ov{A^{-1}} z
\end{equation}
Thus we obtain

\begin{Lemma}

Equation~\ref{equation:linear} implies  equation
equation~\ref{equation:linearsecond} below:
\begin{equation}
\label{equation:linearsecond} A^{-1} \cdot \ov{A^{-1}}
(z)=\frac{1}{|\mu(z)|^2}. z
\end{equation}
\end{Lemma}

Notice that the matrix $B:=A^{-1} \ov {A^{-1}}$ is an hermitian
matrix because $A$ is symmetric. Moreover we have:

\begin{Lemma}
The eigenvalues of $B=A^{-1} \ov {A^{-1}}$ are all positive.
\end{Lemma}
\begin{proof}
Given an eigenvector $\vec u \in \bc^n$ with eigenvalue $\lambda$
we have $B \vec u = \lambda \vec u$ and then $\lg A^{-1} \ov
{A^{-1}} \vec u , \vec u \rg = \lg \lambda \vec u, \vec u \rg =
\lambda \lg \vec u , \vec u \rg$. On the other hand, $\lg A^{-1}
\ov {A^{-1}} \vec u, \vec u \rg = \lg  \ov {A^{-1}} \vec u ,  (\ov
{A^{-1}})^t \vec u \rg = ||  \ov {A^{-1}} \vec u || ^2$. Thus we
have $\lambda || \vec u|| ^2 = || \ov {A^{-1}} \vec u || ^2$ what
implies $\lambda >0$.

\end{proof}

Let now $w^{(1)}\in \Sigma$ be a tangency point, $w^{(1)}\in
S^{n-1}(r)$. We introduce the complex line $\Sigma(w^{(1)})=\{ T.
w^{(1)}: \, T \in \bc\}$.
\begin{Lemma}
The complex line $\Sigma(w^{(1)})$ is contained in the variety of
contacts $\Sigma$.
\end{Lemma}
\begin{proof}
Given a point $T. w^{(1)} \in \Sigma(w^{(1)})$ we have $A (T
w^{(1)})= T A w^{(1)} = T \frac{1}{\ov {\mu(w^{(1)})}}
\ov{w^{(1)}} = \frac{T}{\ov T} . \frac{1}{\ov{\mu(w^{(1)})}} \ov{T
w^{(1)}}$. On the other hand, since $\Omega=\Omega_A$ is linear we
have $\mu(T w^{(1)})=\frac{\lg \vec R(T.w^{(1)}), \grad(\Omega)(T
.w^{(1)}) \rg }{||\grad(\Omega)(T. w^{(1)})||^2} = \frac{T . \ov T
\lg \vec R(w^{(1)}), \grad(\Omega)
(w^{(1)})\rg}{|T|^2.||\grad(\Omega)(w^{(1)})||^2} = \frac{T}{\ov
T} \mu(w^{(1)}).$ Hence we obtain $A (T w^{(1)})=\frac{1}{\ov
{\mu(T. w^{(1)})}} \ov{T.w^{(1)}}$, i.e., $T. {w^{(1)}} \in
\Sigma( {w^{(1)}})$.

\end{proof}

Let now $w^{(2)}\in \Sigma$ be another contact point, say
$w^{(2)}\in S^{2n-1}(r^\prime)$. Suppose that the contact points
$w^{(1)}$ and $w^{(2)}$ are linearly independent. Given a linear
combination $w=T_1 w^{(1)} + T_2 w^{(2)}$ of the contact points we
investigate whether this is also a contact point. Suppose
therefore that $T_1 \ne 0 \ne T_2$ and that $w\in \Sigma$. By
equation (\ref{equation:linear}) we obtain $A(T_1 w^{(1)} + T_2
w^{(2)})= \frac{1}{\ov{\mu (T_1 w^{(1)} + T_2 w^{(2)})}}. \ov{(T_1
w^{(1)} + T_2 w^{(2)})}$. On the other hand, $A(T_1 w^{(1)} + T_2
w^{(2)})= A(T_1 w^{(1)}) + A( T_2 w^{(2)}) = \frac{1}{\ov{\mu(
T_1. w^{(1)})}}. \ov{T_1 .w^{(1)}} + \frac{1}{\ov{\mu( T_2.
w^{(2)})}}. \ov{T_2 .w^{(2)}}$.
 Therefore, by linear independence of $w^{(1)}$ and $w^{(2)}$  we have
 $\mu((T_1 w^{(1)}) = \mu (T_1 w^{(1)} + T_2 w^{(2)}) = \mu(T_2
 w^{(2)})$. Since $\mu(T_j w^{(j)}) = \frac{T_j}{\ov {T_j}}
 \mu(w^{(j)})$we conclude that
$|\mu(w^{(1)})|=\mu(w^{(2)})|$.  Thus

\begin{Lemma}
If $|\mu(w^{(1)})| \ne | \mu(w^{(2)})|$ then the point $T_1
w^{(1)} + T_2 w^{(2)}, T_1 \ne 0, T_2 \ne 0$ is not a contact
point.
\end{Lemma}

Then we have the following:

\begin{Proposition}
\label{Proposition:distincteigenvalueslines} Let $A$ be a $n\times
n$ nonsingular complex symmetric matrix. If the eigenvalues of
$A^{-1} \cdot \ov{A^{-1}}$ are pairwise distinct with eigenvectors
say $w^{(1)},...,w^{(n)}$ then the variety of contacts
$\Sigma({\Omega _A}, \vr)$, is the union of the $n$ lines
$\Sigma(w^{(j)}), \, j=1,...,n$.
\end{Proposition}

Denote by: \begin{itemize}

\item $\Sim(n)$ the subspace of symmetric $n\times n$ complex
matrices.

\item  $\Sim (n)^* \subset \Sim(n)$ the open subset of invertible
symmetric matrices.

\item  $\mathcal M(n)\subset \Sim(n)^*$ the set of Morse type
symmetric invertible matrices.

\item   $\mathcal M{\mathbb R} S(n)\subset \Sim(n)^*$  the set of
all invertible symmetric {\em real} matrices having eigenvalues
$\lambda_i (1 \leq i \leq n)$ such that $\lambda_i ^2 \ne
\lambda_j ^2, \forall i \ne j$.

\end{itemize}

\begin{Proposition}
\label{Propostion:zariskidense} We have $\mathcal M(n)\subset
\Sim(n)^*\subset \Sim(n)$ and the inclusions are dense.

\end{Proposition}
\begin{proof}
First we observe that $\Sim (n)^*$ is a Zariski open subset of the
affine manifold (affine vector space) $\Sim (n)$, it is the
complement of the proper Zariski closed set (algebraic
submanifold) $\{A \in \Sim(n) : \det A=0\}$. We note that
$\mathcal M(n)$ contains $\mathcal M \mathbb RS(n)$. The
complement $\Sim(n)^* \setminus \mathcal M(n)$ is characterized by
the fact that for $A\in \Sim(n)^*\setminus \mathcal M(n)$ the
hermitian matrix $A^{-1} . \ov{A^{-1}}$ has some multiple
eigenvalue if and only if its characteristic polynomial
$m_{A^{-1}\ov{A^{-1}}}(\lambda)\in \bc [\lambda]$ has some
multiple zero. This is an algebraic condition on its coefficients.
This argumentation shows that $\Sim(n)^*\setminus \mathcal M(n)$
is a closed Zariski (algebraic) subset of $\Sim(n)^*$. Thus
$\mathcal M(n)$ is dense in $\Sim(n)^*$ which is dense in
$\Sim(n)$.
\end{proof}

\begin{Definition}
{\rm A $n\times n$ symmetric complex matrix $A$ will be called of
{\it Morse type} if the corresponding foliation $\Omega _A$ is of
Morse type.}
\end{Definition}

\begin{proof}[Proof of Theorem~\ref{Theorem:linearapprox}]
Theorem~\ref{Theorem:linearapprox} follows from
Proposition~\ref{Proposition:distincteigenvalueslines} and
~\ref{Propostion:zariskidense}.
\end{proof}

\noindent{\bf Acknowledgement}: This research work was done during
a visit of the first named author to the Instituto Nacional de
Matemática Pura e Aplicada (IMPA), Rio de Janeiro, Brazil. The
authors want to thank IMPA for the hospitality.

\bibliographystyle{amsalpha}

\vglue.2in

\begin{tabular}{ll}
Toshikazu Ito  & \qquad  Bruno Sc\'ardua\\
Department of Natural Science  & \qquad Inst. Matem\'atica\\
Ryukoku University  & \qquad Universidade Federal do Rio de Janeiro\\
Fushimi-ku, Kyoto 612 & \qquad  Caixa Postal 68530\\
JAPAN   & \qquad 21.945-970 Rio de Janeiro-RJ\\
&  \qquad BRAZIL
\end{tabular}

\end{document}